\documentclass[11pt]{amsart}
\usepackage{latexsym}
\usepackage{graphicx}
\usepackage{subfigure}
\usepackage{psfrag}
\usepackage{amsmath}
\usepackage{amssymb}
\usepackage{mathrsfs}
\usepackage{epstopdf}
\usepackage[all]{xypic}
\DeclareMathOperator{\Ker}{Ker}

\DeclareMathOperator{\Coker}{Coker}

\newtheorem{theorem}{Theorem}[section]
\newtheorem{lemma}[theorem]{Lemma}
\newtheorem{prop}[theorem]{Proposition}
\newtheorem{prob}[theorem]{Problem}
\newtheorem{corollary}[theorem]{Corollary}

\newenvironment{remark}{\noindent \textbf{Remark}.}{\hfill $\square$}

\renewcommand{\Im}{\mathop{\rm Im}\nolimits}

\newcommand{\bb}{\mathbb}
\newcommand{\cal}{\mathcal}

\numberwithin{equation}{section}

\title{Some results on the generic vanishing of Koszul cohomology via deformation theory}
\author{Jie Wang}
\address{Department of Mathematics, University of Georgia\\
 Athens GA, 30602.}

\email{jiewang@math.uga.edu}
\subjclass[2000]{}
\begin{document}
\maketitle

\begin{abstract} We study the deformation-obstruction theory of Koszul cohomology groups of $g^r_d$'s on singular nodal curves. We compute the obstruction classes for Koszul cohomology classes on singular curves to deform to a smooth one. In the case the obstructions are nontrivial, we obtain some partial results for generic vanishing of Koszul cohomology groups.
\end{abstract}

\bigskip

\section{Introduction.}

In this paper we apply deformation theory to study the syzygies of general curves in $\mathbb{P}^r$ with fixed genus and degree.
Let $L$ be a base point free $g^r_d$ on a smooth curve $X$, the Koszul cohomology group $K_{p,q}(X,L)$ is the cohomology of the Koszul complex at $(p,q)$-spot
$$\xymatrix{\ar[r]&\wedge^{p+1}H^0(L)\otimes H^0(X,L^{q-1})\ar[r]^-{d_{p+1,q-1}}&\wedge^pH^0(L)\otimes H^0(X,L^q)\ar[r]^-{d_{p,q}}&\wedge^{p-1}H^0(L)\otimes H^0(X,L^{q+1})\ar[r]&}$$
where
$$d_{p,q}(v_1\wedge...\wedge v_p\otimes \sigma)=\sum_{i}(-1)^iv_1\wedge...\wedge\widehat v_i\wedge..\wedge v_p\otimes v_i\sigma.$$ 

Koszul cohomology groups $K_{p,q}(X,L)$ completely determine the shape of a minimal free resolution of the section ring
$$R=R(X,L)=\bigoplus_{k\ge0}H^0(X,L^k).$$
and therefore carry a lot of information of the extrinsic geometry of $X$.

We are interested in Green's question

\begin{prob}\label{green}(Green)
What is the variational theory of the $K_{p,q}(X,L)$? What do they look like for $X$ a general curve and $L$ a general $g^r_d$? 
\end{prob}

If $(X,L)$ is general in $\cal{G}^r_{g,d}$ (in this paper, this means the Brill-Noether number $\rho=g-(r+1)(g-d+r)\ge0$ and $(X,L)$ is a general point of the unique component of $\cal{G}^r_{g,d}$ which dominates $\cal{M}_g$), it is well known that we only have to determine $K_{p,1}(X,L)$, or equivalently $K_{p-1,2}(X,L)$, for $1\le p\le r-1$ (c.f. Section 2).


Problem \ref{green} seems to be too difficult to answer in its full generality.
For arbitrary $g^r_d$ on a general curve $X$,  the first case to determine $K_{1,1}(X,L)$ or $K_{0,2}(X,L)$ is still open. The Maximal Rank Conjecture (MRC) \cite{EH2} predicts that the multiplication map
\begin{eqnarray}\label{multimap}Sym^2H^0(X,L)\stackrel{\mu}\longrightarrow H^0(X,L^2)
\end{eqnarray}
 is either injective or surjective i.e. 
\begin{eqnarray}\label{MRC}\min\{k_{1,1}(X,L),k_{0,2}(X,L)\}=0.
\end{eqnarray}

 

Geometrically, this means that the number of quadrics in $\mathbb{P}^r$ containing $X$ is as simple as the Hilbert function of $X\subset\bb{P}^r$ allows.

 There are many partial results about (\ref{MRC}) using the so-called \textquotedblleft m$\acute{e}$thode d'Horace"
originally proposed by Hirschowitz. It amounts to a degeneration argument to a carefully chosen 
singular curve in projective space and proving the statement on such a curve by a delicate
inductive argument. We refer to, for instance, \cite{BF1}, \cite{BF2} for some recent results in this direction.

For higher syzygies, again there are many results (c.f. \cite{A1}, \cite{B}, \cite{Ein}, and \cite{F}). One breakthrough result is Voisin's solution to the generic Green's conjecture \cite{V1} \cite{V2}, which solves Problem \ref{green} for the case $L=K_X$.



For the vanishing of $K_{p,1}$, there is the work of Aprodu \cite{A1} \cite{A2}, which proved the generic version of the Green-Lazarsfeld Gonanity Conjecture. This conjecture predicts that for smooth curve $X$ of gonanity $d$, and $L$ a sufficiently positive line bundle on $X$,
$$K_{h^0(L)-d,1}(X,L)=0.$$
Note that Problem \ref{green} does not have any assumption on the positivity of $L$.

It seems that the method of all of the above results amount to degenerating to special curves, often a carefully chosen singular one, and verifying the statements on these special curves. Given the fact that sometimes such special curves are difficult to find, and the inductive arguments could get technical, we would like to take a slightly different point of view. We will consider one parameter degeneration to the simplest possible singular curves, namely union of two smooth curves meeting at a node. Of course, there is no hope to directly verify the vanishing statements we would like to prove on these curves (c.f. section 3), but we are able to compute the obstructions for the 'extra' Koszul classes of the singular fiber to deform to nearby fibers. If one could prove these 'extra' Koszul classes are obstructed, we conclude the general fiber has the vanishing property we need.  We feel this point of view has a good chance to generalize.

More precisely,  suppose property ${\bf{GV}}(p)^r_{g,d}$ holds, i.e.  for general $L'=g^r_d$ on general curve $C$ of genus $g$ we have
\begin{eqnarray} \label{gv}\min\{ k_{p,1}(C,L'), k_{p-1,2}(C,L')\}=0.
\end{eqnarray}
We ask the following question
\begin{prob}\label{question}In what situation does ${\bf{GV}}(p)^r_{g+1,d+1}$ hold? \end{prob}

If the answer to this question is Yes, then one could set up an inductive argument. Each step $r$ is fixed and $g$, $d$ go up by $1$, or equivalently, $r$ and $h^1$ fixed, $g$ goes up by $1$.

In the case $p=1$, the Maximal Rank Conjecture predicts the answer should always be affirmative. For higher syzygies, it is not always the case, but one would like to prove some generic vanishing results for some special $\{g,r,d\}$.

In this paper, we give a simple condition to guarantee ${\bf GV}(p)^r_{g,d}$ implies ${\bf GV}(p)^r_{g+1,d+1}$ from a deformation-theoretic point of view.  We study the deformation theory of Koszul cohomology groups on the simplest kind of singular curve $X_0$: a union of a general curve $C$ of genus $g$ and an elliptic curve $E$ meeting at a node $u$. $L_0$ is carefully chosen (c.f section 3) such that 
\begin{enumerate}
\item $(X_0,L_0)$ is smoothable to $L_t=g^r_{d+1}$ on a smooth curve $X_t$ of genus $g+1$,
\item $L_0|_C=L'$ and therefore $\min\{ k_{p,1}(C,L_0|_C), k_{p-1,2}(C,L_0|_C)\}=0$.
\item $L_0|_E=\cal{O}_E(v)$ for another general point $v\in E$. 

\end{enumerate}

We prove that
\begin{theorem}\label{generalcase}Let $C\subset\bb{P}^r$ be a general curve, $|L'|$ a general $g^r_d$ on $C$ and $M_{L'}$ be the kernel bundle defined by the sequence
$$\xymatrix{0\ar[r]&M_{L'}\ar[r]&H^0(L')\otimes\cal{O}_C\ar[r]^-{ev}&L'\ar[r]&0}.$$
then the following holds
\begin{enumerate}
\item If $K_{p,1}(C,L')=0$ then $K_{p,1}(X_t,L_t)=0$.
\item If $K_{p-1,2}(C,L')=0$ and 
\begin{eqnarray}\label{h0condition}h^0(C,\wedge^{r-p}M_{L'}\otimes K_C)=h^0(C,\wedge^{r-p}M_{L'}\otimes K_C(2u))
\end{eqnarray}
for a general point $u\in C$, then $K_{p-1,2}(X_t,L_t)=0$.

\end{enumerate}
In other words, ${\bf GV}(p)^r_{g,d}$ always implies ${\bf GV}(p)^r_{g+1,d+1}$ if (\ref{h0condition}) holds.

\end{theorem}

  The upshot is that under such degeneration, we could explicitly compute generators of $K_{p,q}(X_0,L_0)$. Unfortunately $(X_0,L_0)$ does not satisfy (\ref{gv}). However, we could compute the obstructions for the \textquotedblleft extra"  Koszul classes to deform to $K_{p,q}(X_t,L_t)$. If every \textquotedblleft extra" Koszul class is obstructed, we conclude that (\ref{gv}) holds for $(X_t,L_t)$. Condition (\ref{h0condition}) is a sufficient condition for the \textquotedblleft extra"  Koszul classes to be obstructed.

In the case $p=1$ (Maximal Rank Conjecture) this sufficient condition turns out to be very geometric:
\begin{theorem} \label{mainresult}Let $C\subset\bb{P}^r$ be a general curve embedded by a general $g^r_d$ $|L'|$, and suppose one of the following two conditions holds
\begin{enumerate}
\item $\mu$ in (\ref{multimap}) is injective, or
\item $\mu$ is surjective and there exists a quadric $Q\in \Ker(\mu)$ containing $C$ but not containing the tangential variety $TC:=\cup_{u\in C}T_uC$,
\end{enumerate}
 then $(MRC)^r_{g+1,d+1}$ holds as well.

\end{theorem}

To apply theorem \ref{mainresult} to the Maxiaml Rank Conjecture, one has to verify a hypothesis in (b) which seems geometrically interesting in its own right. Hopefully there will be some other applications.

Starting form rational normal curves and canonical curves are projectively normal, we verify hypothesis (b) in some special cases and get some partial results: 

 \begin{corollary}\label{bestresult}
 Let $(X,L)$ be a general pair in $\cal{G}^r_{g,d}$ with $h^1(L)\le1$. Suppose
 $$d>\frac{5}{4}g+\frac{9}{4},\ \ \text{if}\ h^1(L)=0, \ \text{or}$$
 $$d>\frac{5}{4}g+\frac{3}{4},\ \ \text{if}\ h^1(L)=1,$$
 then $(X,L)$ is projectively normal.
 \end{corollary}
 
 It is a very well known result of Green-Lazarsfeld \cite{GL1} that any very ample line bundle $L$ on $X$ with
 \begin{eqnarray}\label{bound}deg(L)\ge2g_X+1-2h^1(L)-Cliff(X)
 \end{eqnarray}
 is projectively normal and the bound is sharp. Notice that (\ref{bound}) implies that $h^1(L)\le1$.
 
 If $X$ is general,
 $$Cliff(X)=\lfloor\frac{g_X-1}{2}\rfloor,$$ thus Green-Lazarsfeld theorem predicts projective normality for general curves if $d$ is bigger than roughly $\frac{3}{2}g$. Corollary \ref{bestresult} thus says that if $L$ is also general, we could improve the lower bound of $d$ to roughly $\frac{5}{4}g$. 
 
 The bounds in Corollary \ref{bestresult} is weaker than the bounds in \cite{BF2}.

 We could also fix a small $r$ and let $h^1$ to be arbitrarily large. This is 
 
\begin{corollary}\label{theorem1.6} The maximal rank conjecture (for quadrics) holds if $r\le4$.
\end{corollary}
The reason we can get rid of the restriction on degree of the line bundle for small $r$ is because we can always verify the hypothesis on $TC$ in Theorem \ref{mainresult} (b) if $r\le4$. Thus $(MRC)^r_{g,d}$ always imply $(MRC)^r_{g+1,d+1}$.

For higher syzygies, we do not expect analogously $\min\{k_{p,1},k_{p-1,2}\}=0$ for $p\ge2$. We refer the audience to section $2$ for a counterexample. Nevertheless, we do wish to  to obtain certain vanishing results or effective upper bounds on $k_{p,q}$.

The difficulty to generalize the inductive argument to higher syzygies is two-fold. First, there are relatively few known cases to start the induction with. There is essentially a single known starting series of examples for vanishing of syzygies, namely Voisin's solution to the generic Green conjecture. Besides Voisin's theorem, Farkas \cite{F2} proved that properties ${\bf{GV}}(2)^7_{16,21}$ and ${\bf{GV}}(3)^{10}_{22,30}$ hold. Secondly, for higher syzygies, the sufficient condition for \textquotedblleft extra" Koszul classes to be obstructed  is not as geometric.

Nevertheless, we summarize our results on higher syzygies as below
\begin{theorem}\label{badresult}Let $X$ be a general curves of genus $g$, $L$ be a general $g^r_d$ on $X$. Then
\begin{enumerate}
\item If $g\ge r+1$, $K_{p,1}(X,L)=0$ for $p\ge\lfloor\frac{r+1}{2}\rfloor$.\\

\item If $h^1(L)=1$ (which implies that $g\ge r+1$), 

$$K_{p-1,2}(X,L)=0\  \text{for}\  1\le p\le r-\lfloor\frac{g}{2}\rfloor,\ \text{and}$$
 $$k_{p-1,2}(X,L)\le(g-2r+2p-1){r-1\choose p-1}\ \text{for}\ p>r-\lfloor\frac{g}{2}\rfloor.$$
\end{enumerate}
\end{theorem}

Combining Corrollary \ref{bestresult} and \ref{badresult} (a), we can determine table $1$ for general $g^r_d$ with $r\le4$:

\begin{corollary} For general pair $(X,L)$ in $\cal{G}^r_{g,d}$ with $r\le4$, $g\ge r+1$,
$$\min\{k_{p,1}(X,L),k_{p-1,2}(X,L)\}=0.$$
\end{corollary}

The organization of this paper is as follows. In Section 2, we review some basic facts on Koszul cohomology of general curves. In Section 3, we study the Koszul cohomology of the central fiber $(X_0,L_0)$. We explicitly write down the generators of the \textquotedblleft extra\textquotedblright \ Koszul classes in $K_{r-p,0}(X_0,L_0;\omega_{X_0})\cong K_{p-1,2}(X_0,L_0)^{\lor}$. Section 4 contains a computation of the obstructions for these classes to deform and Section 5 gives a sufficient condition for the obstruction classes to be linearly independent and a proof of Theorem \ref{generalcase}. In Section 6, we focus on $p=1$ case, we prove Theorem \ref{mainresult} and Corollaries \ref{bestresult} and 1.8. 
Finally in Section $7$, we consider higher syzygies for line bundles with $h^1=1$. In some special range of $p$, we are able to prove some vanishing results as in Theorem \ref{badresult}.

{\bf Acknowledgements.} This work is a continuation of my thesis project. I would like to thank my thesis advisor Herb Clemens for suggesting the problem and method and his constant support on this work. 
I would also like to thank Aaron Bertram, Gavril Farkas and Joe Harris for generously sharing their ideas on this problem. Last but not the least, I thank the referee for the helpful comments and suggestions to improve the paper.

\section{ Koszul Cohomology of general curves}

 We first summarize several special properties of Koszul cohomology groups on general curves over $\mathbb{C}$. We refer to \cite{AN} and \cite{E} for general facts about Koszul cohomology.
 
 \begin{prop}Suppose $X$ is a general curve, and $L$ is a complete $g^r_d$ on $X$, then the following holds
 \begin{enumerate}
  \item $K_{p,0}(X,L)=0\text{\ except\  when\ }p=0 \text{\ and\ } k_{0,0}(X,L)=1$;
 \item $K_{p,q}(X,L)=0\text{\ for\ }q\ge4$;
 \item $K_{p,3}(X,L)=0\text{\ except\  when\ }p=r-1 \text{\ and\ }  k_{r-1,3}(X,L)=h^1(L)$.
 \end{enumerate}
 \end{prop}
 \begin{proof} Statement (a) follows from the definition of Koszul cohomology.

To prove (b) and (c), we use the following facts.
 \begin{enumerate}
\item [i)] The multiplication map
$$H^0(X,L)\otimes H^0(X,K_X\otimes L^{-1})\longrightarrow H^0(X,K_X)$$
is injective. This is the Gieseker-Petri theorem.

\item [ii)] $$H^0(X,K_X\otimes L^{-2})=0.$$
\end{enumerate} This is a direct consequence of i) (c.f. \cite{AC}).

Statement (b) follows from ii) and the duality theorem (c.f. \cite{AN} Sec. 2.3) of Koszul cohomology
\begin{eqnarray}\label{dual}
K_{p,q}(X,L)=K_{r-1-p,2-q}(X,L;K_X)^{\lor}.
\end{eqnarray}
To prove (c), we first apply (\ref{dual}) and note that the Koszul differential $d_{r-1-p,-1}$ factors as
$$\xymatrix{\bigwedge^{r-1-p}H^0(L)\otimes H^0(K_X\otimes L^{-1})\ar[r]^-{d_{r-1-p,-1}}\ar[d]^-{\lrcorner\otimes Id}&\bigwedge^{r-2-p}H^0(L)\otimes H^0(K_X)\\
\bigwedge^{r-2-p}H^0(L)\otimes H^0(L)\otimes H^0(K_X\otimes L^{-1})\ar[ur]^-{Id\otimes\mu}}.$$
By i) both $\lrcorner\otimes Id$ and $Id\otimes\mu$ are injective.
\end{proof}

As a consequence, we have
\begin{corollary}\label{cor2.2} Let $X$ be a general curve, $L$ be a globally generated $g^r_d$ with $r\ge1$. 

\begin{enumerate}
\item $L$ is normally generated if and only if the multiplication map
$$\mu: S^2H^0(X,L)\longrightarrow H^0(X,L^2)$$
is surjective.
\item If $L$ is normally generated, the homogeneous ideal $I_X$ is generated by quadrics and cubics.
\end{enumerate}
\end{corollary}
\begin{proof}
The only possible nonzero $K_{0,q}$ for $q\ge2$ is $K_{0,2}(X,L)=\Coker(\mu)$. If $K_{0,2}(X,L)=0$, $L$ is normally generated. Since $k_{1,q}$ is the number of minimal generators of $I_X$ of degree $q+1$, (b) follows.
\end{proof}

Moreover, since taking cohomology does not change the Euler characteristic of the complex, we have for any $1\le p\le r-1$,
\begin{eqnarray}k_{p,1}(X,L)-k_{p-1,2}(X,L)\nonumber&=&\sum_{i+j=p+1}(-1)^{j+1}\ \text{dim}_{\mathbb{C}}((\bigwedge^{i}V)\otimes H^0(X,L^{j}))\nonumber\\&=&{r+1\choose p}(g-d+r)-{r+1\choose p+1}g+{r-1\choose p}d+{r\choose p+1}(g-1)\nonumber
\end{eqnarray}

Denote this number $b_p(X,L)$, which only depends on $g$, $r$, $d$, $p$. Therefore, to determine the Koszul cohomology of $(X,L)$, it suffices to determine either row $q=1$ or $q=2$.

\begin{remark}
 Based on the maximal rank conjecture, one might expect that analogously 
 \begin{eqnarray}\label{nottrue}
 \min\{k_{p,1}(X,L),k_{p-1,2}(X,L)\}=0
 \end{eqnarray}
  for general $(X,L)$. But this is not the case. In fact, F. Schreyer proved in his thesis (c.f Green \cite{Gr} (4.a.2) for more details) that any curve $X$ of genus $g$, there exists a number $d_0$ such that if $\deg(L)=d\ge d_0$, then
$$K_{p,2}(X,L)\ne0\ \text{if}\ r-1\ge p\ge r-g.$$

On the other hand, it follows from a theorem of Green and Lazarsfeld (c.f  \cite{GL84} or \cite{AN} Corollary 3.39) that for $d$ large,
$$K_{p,1}(X,L)\ne 0\ \text{if}\ 1\le p\le r-\lfloor\frac{g}{2}\rfloor-2.$$
Thus for $r-g+1\le p\le r-\lfloor\frac{g}{2}\rfloor-2$, we do not have (\ref{nottrue}).
\end{remark}

\section{Koszul cohomology of the central fiber}

Let $L'$ be a $g^r_d$ on a smooth curve $C$ of genus $g$ and $X_0=C\cup E$ be the reducible nodal curve consisting of $C$ and an elliptic curve $E$ meeting at a general point $u$. Let $L_0$ be the line bundle on $X_0$ such that 
$$L_0|_C=L',$$
and
$$L_0|_E=\cal{O}_E(v)$$
where $v\ne u$. We would like to study the relations between $K_{p,q}(C,L')$ and $K_{p,q}(X_0,L_0)$ in this section.

First observe that by construction, any (global) section of $L'$ on $C$ extends uniquely to a section of $L_0$ on $X_0$, thus we have a natural isomorphism $\phi: H^0(C,L')\cong H^0(X_0,L_0)$. Moreover, by Riemann-Roch, $h^1(C,L')=h^1(X_0,L_0)$, and there is a natural identification
$$H^0(C,K_C\otimes L'^{-1})\cong H^0(X_0, \omega_{X_0}\otimes L_0^{-1}).$$ 

A first consequence is

 \begin{prop} \label{prop3.1}If $K_{p,1}(C,L')=0$, then $K_{p,1}(X_0,L_0)=0$.
 \end{prop}
\begin{proof}
Consider the following commutative diagram
$$\xymatrix{\bigwedge^{p+1}H^0(L_0)\ar[r]\ar[d]^-{\cong}&\bigwedge^{p}H^0(L_0)\otimes H^0(L_0)\ar[r]\ar[d]^-{\cong}&\bigwedge^{p-1}H^0(L_0)\otimes H^0(L_0^2)\ar[d]\\
\bigwedge^{p+1}H^0(L')\ar[r]&\bigwedge^{p}H^0(L')\otimes H^0(L')\ar[r]&\bigwedge^{p-1}H^0(L')\otimes H^0(L'^2)}
$$
where the vertical arrows are restriction maps to $C$. The hypothesis says that the lower row is exact in the middle, a simple diagram chasing gives the conclusion.
\end{proof}

The argument in proposition \ref{prop3.1} does not generalize to the case $q=2$ because $H^0(C,L'^2)$ is not isomorphic to $H^0(X_0,L_0^2)$. Instead, we dualize using (\ref{dual})
$$K_{p-1,2}(C,L')^{\lor}\cong K_{r-p,0}(C,L';K_{C})$$
and compare $K_{r-p,0}(C,L';K_{C})$ with $K_{r-p,0}(X_0,L_0;\omega_{X_0})$.

Here $\omega_{X_0}$ is the dualizing sheaf of $X_0$. Its restriction to $C$ and $E$ are line bundles $K_C(u)$ and $K_E(u)$ respectively. A global section of the dualizing sheaf consists of (global) one forms on $C$ and $E$, viewed as sections of $K_C(u)$ and $K_E(u)$ which vanish at $u$ respectively.

\vspace{2cm}

\begin{figure}[h]
 \psfrag{L'}{ $L'$}
  \psfrag{O(v)}{ $\mathcal{O}_E(v)$}
\psfrag{KC(u)}{ $K_C(u)$}
\psfrag{KE(u)}{ $K_E(u)$}
\psfrag{KCL0-1}{ $K_C(u)\otimes L'^{-1}$}
\psfrag{KE(u-v)}{ $K_E(u-v)$}
\psfrag{KCL'(u)}{ $K_C(u)\otimes L'$}
\psfrag{KE(u+v)}{$K_E(u+v)$}
\psfrag{L0}{$L_0$}
\psfrag{Omega}{$\omega_{X_0}$}
\psfrag{OmegaL0-1}{$\omega_{X_0}\otimes L_0^{-1}$}
\psfrag{OmegaL0}{$\omega_{X_0}\otimes L_0$}

   \includegraphics[scale=0.6]{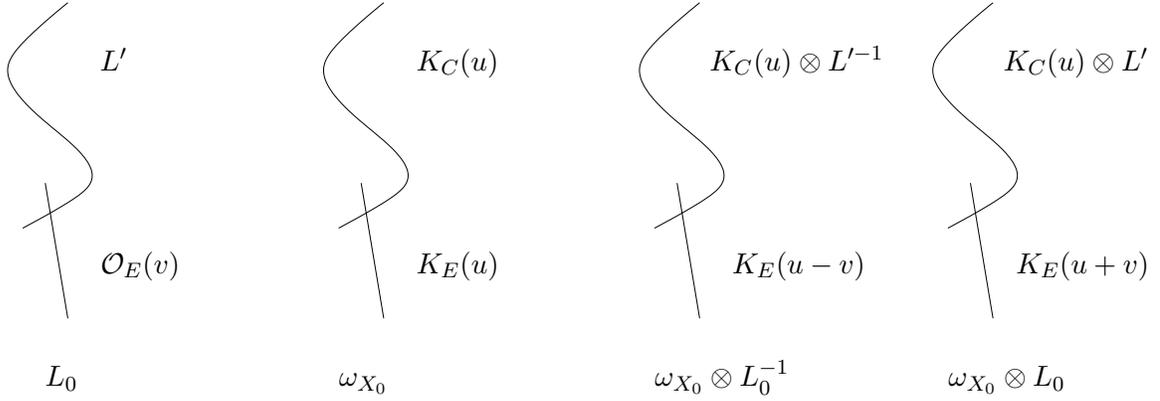}
  \caption{The line bundles on the central fiber }
  \label{figure1}
\end{figure}

\vspace{2cm}

 Figure \ref{figure1} describes the various line bundles in question on $X_0$ and their restrictions to each components. The S-shaped curve is $C$ and the straight line is $E$.
  
{\bf Choose $\{\omega_0,...,\omega_{g-1}\}$ a basis of $H^0(C,K_C)$ and $\{\omega_{g}\}$ a basis of $H^0(E,K_{E})$. We will think of $\omega_i$ for $0\le i\le g-1$ as sections in $H^0(K_C(u))$ which vanishes on $u$ and then extend it over $E$ by the zero section.  Such a section belongs to $H^0(\omega_{X_0})$ and we still denote it as $\omega_i$. Similarly we obtain $\omega_g\in H^0(\omega_{X_0})$ with $\omega_g|_C=0$.}

In this way, we obtain a natural identification
$$\psi: H^0(C,K_C)\oplus H^0(E,K_{E})\cong H^0(C,K_C(u))\oplus H^0(E,K_{E}(u))\cong H^0(X_0,\omega_{X_0}).$$
and 
$$H^0(X_0,\omega_{X_0})=span\{\omega_i|\ i=1,...,g\}$$
Notice also that every section in $H^0(X_0,\omega_{X_0})$ vanish at $u$.

Now suppose $K_{r-p,0}(C,L';K_{C})=0$, we would like to show that $K_{r-p,0}(X_0,L_0;\omega_{X_0})$ can be generated by pure tensors in
$$\wedge^{r-p}H^0(X_0,L_0)\otimes H^0(X_0,\omega_{X_0}).$$

To this end, consider the following commutative diagram
\begin{eqnarray}\nonumber
\xymatrix{\wedge^{r-p+1}H^0(L')\otimes H^0(K_C\otimes L'^{-1})\ar[r]\ar[d]^{\cong}&\wedge^{r-p}H^0(L')\otimes H^0(K_C)\ar[r]\ar[d]^{\phi\otimes\psi}&\wedge^{r-p-1}H^0(L')\otimes H^0(K_C\otimes L')\ar[d]\\
\wedge^{r-p+1}H^0(L_0)\otimes H^0(\omega_{X_0}\otimes L_0^{-1})\ar[r]^-{\delta_{-1}}&\wedge^{r-p}H^0(L_0)\otimes H^0(\omega_{X_0})\ar[r]^-{\delta_0}&\wedge^{r-p-1}H^0(L_0)\otimes H^0(\omega_{X_0}\otimes L_0)}
\end{eqnarray}
the left vertical arrow is an isomorphism because any section in $H^0(X_0, \omega_{X_0}\otimes L_0^{-1})$ restricts to zero on the $E$ component.

Now let $\{\sigma_0,...,\sigma_r\}$ be a basis of $H^0(C,L')$. Extend $\sigma_k$ uniquely to $X_0$ to form a basis of $H^0(X_0,L_0)$, still denoting them by $\sigma_k$. 

We can write any element in $\Ker(\delta_0)$ as
$$\sum_{k_1,...,k_{r-p},j\le g-1}\alpha_{k_1,...,k_{r-p},j}\ \sigma_{k_1}\wedge...\wedge\sigma_{k_{r-p}}\otimes\omega_j+\sum_{k_1,...,k_{r-p}}\beta_{k_1,...,k_{r-p}}\ \sigma_{k_1}\wedge...\wedge\sigma_{k_{r-p}}\otimes\omega_{g}$$

Since the image under $\delta_0$ of the second term $\beta$ restrict to $0$ on $C$ (since $\omega_g$ does), so does the image of the first term $\alpha$. By our assumption, the top row of the above diagram is exact in the middle and therefore  $\alpha\in\Im(\delta_{-1})$. 

We conclude that, 
$$\sum_{k_1,...,k_{r-p}}\beta_{k_1,...,k_{r-p}}\sigma_{k_1}\wedge...\wedge\sigma_{k_{r-p}}\otimes\omega_{g}\in\Ker(\delta_0)$$
and this can happen only if 
 $$\sum_{k_1,...,k_{r-p}}\beta_{k_1,...,k_{r-p}}\sigma_{k_1}\wedge...\wedge\sigma_{k_{r-p}}\in\bigwedge^{r-p}V,$$
 where $V\subset H^0(X_0,L_0)$ is the codimension one subspace consisting of sections which restrict to zero on $E$. Also it is easy to see that a basis of 
$$\bigwedge^{r-p}V\otimes \bb{C}\cdot\omega_g$$
 are linearly in dependent even modulo $\Im(\delta_{-1})$.

We have proven
\begin{lemma} \label{lemma3.2}Notation as above, if $K_{r-p,0}(C,L';K_C)=0$, we have an isomorphism 
$$\xymatrix{K_{r-p,0}(X_0,L_0;\omega_{X_0})\ar[r]^-{\cong}&\bigwedge^{r-p}V\otimes \bb{C}\cdot\omega_g.}$$

\end{lemma}

\section{Infinitesimal calculations}

In this section, we carry out the computation of first order obstructions described in the introduction. We will use the deformation theory of complexes which was developed in \cite{GL}. The general set up is as below.

Let $S$ be a smooth variety, and $F^{\bullet}$ be a bounded complex of locally free sheaves on $S$:
$$\xymatrix{...\ar[r]&F^{p+1}\ar[r]^-{d_{p+1}}&F^p\ar[r]^-{d_p}&F^{p-1}\ar[r]&...}$$
Given a point $t\in S$, denote $F^{\bullet}(t)$ the complex of vector spaces at $t$ determined by the fibers of $F^{\bullet}$, i.e.
$$F^{\bullet}(t)=F^{\bullet}\otimes\mathbb{C}(t),$$
where $\mathbb{C}(t)$ is the residue field of $S$ at $t$. 

The deformation theory of $H^i(F^{\bullet}(t))$ as $t$ moves near $0\in S$ is controlled by the derivative complex, which associates a tangent vector $v\in T_0S$ a complex:
$$\xymatrix{...\ar[r]&H^{p+1}(F^{\bullet}(0))\ar[r]^-{D_v(d_{p+1})}&H^p(F^{\bullet}(0))\ar[r]^-{D_v(d_p)}&H^{p-1}(F^{\bullet}(0))\ar[r]&...}$$ 

A (co)homology class $[c]\in H^p(F^{\bullet}(0))$ deforms to first order along $v$ if and only if $D_v(d_p)([c])=0\in H^{p-1}(F^{\bullet}(0))$.

To describe the $D_v(d_p)$, recall that a tangent vector $v\in T_0S$ corresponds to an embedding of the dual numbers $D$ into $S$. So one gets a short exact sequence
$$\xymatrix{0\ar[r]&\mathbb{C}(0)\ar[r]&D\ar[r]&\mathbb{C}(0)\ar[r]&0}.$$
Tensoring the sequence with $F^{\bullet}$ yields a short exact sequence of complexes,
which in turn gives rise to connecting homomorphisms
$$\xymatrix{H^p(F^{\bullet}\otimes\mathbb{C}(0))\ar[r]^-{D_v(d_p)}\ar@{=}[d]&H^{p-1}(F^{\bullet}\otimes\mathbb{C}(0))\ar@{=}[d]\\H^p(F^{\bullet}(0))&H^{p-1}(F^{\bullet}(0))&}$$
One checks that $D_v(d_{p})\circ D_v(d_{p+1})=0$.

Now Let $(X_0,L_0)$ be the pair constructed in the previous section. We will further assume that both $(C,L')$ and the crossing point $u$ are general. $(X_0,L_0)$ determines a limit linear series in the sense of Eisenbud and Harris \cite{EH1}. By counting Brill-Noether numbers, it is easy to see this limit linear series is deformable to general pairs $(X_t,L_t)$. So let $\mathcal{L}\rightarrow\mathcal{X}\rightarrow\Delta$ be the total space of an one parameter family of general pairs $(X_t,L_t)\in\cal{G}^r_{g+1,d+1}$ degenerating to $(X_0,L_0)$. We will apply the deformation theory described above to the Koszul complex computing $K_{r-p,0}(X_t,L_t;\omega_{X_t})$:
$$\xymatrix{\wedge^{r-p+1}H^0(L_t)\otimes H^0(\omega_{X_t}\otimes L_t^{-1})\ar@{^{(}->}[r]&\wedge^{r-p}H^0(L_t)\otimes H^0(\omega_{X_t})\ar[r]^-{\delta_t}&\wedge^{r-p-1}H^0(L_t)\otimes H^0(\omega_{X_t}\otimes L_t).}$$

By Gieseker-Petri theorem, the left arrow is injective for all $t$ (even at time zero), so $k_{r-p,0}(X_t,L_t;\omega_{X_t})$ can only go up at $t=0$ if $\Ker(\delta_t)$ does. We would like to compute the derivative of $\delta_t$ at $t=0$:
\begin{eqnarray}\label{derivative}
\xymatrix{K_{r-p,0}(X_0,L_0;\omega_{X_0})\ar[r]^-{D(\delta_t)|_{t=0}}&K_{r-p-1,1}(X_0,L_0;\omega_{X_0}).}
\end{eqnarray}

To illustrate the idea, let us first take a look at the baby case when $p=r-1$ (On the other hand, the main case we are interested in is the case $p=1$). The general case is just notationally more complicated. In this special case, the Koszul differential $\delta_t$ becomes the multiplication map $\mu_t$
$$\xymatrix{\wedge^{2}H^0(L_t)\otimes H^0(\omega_{X_t}\otimes L_t^{-1})\ar@{^{(}->}[r]&H^0(L_t)\otimes H^0(\omega_{X_t})\ar[r]^-{\mu_t}& H^0(\omega_{X_t}\otimes L_t)}$$
 and the derivative map is
 $$\xymatrix{K_{1,0}(X_0,L_0;\omega_{X_0})\ar[r]^-{D(\mu_t)|_{t=0}}&K_{0,1}(X_0,L_0;\omega_{X_0}).}$$

For simplicity, denote $\omega_g\in H^0(X_0,\omega_{X_0})$ by $\omega$. By Lemma \ref{lemma3.2}, if $K_{1,0}(C,L',K_C)=0$, we have
$$\xymatrix{K_{1,0}(X_0,L_0;\omega_{X_0})\ar[r]^-{\cong}&V\otimes \bb{C}\cdot\omega\subset V\otimes H^0(X_0,\omega_{X_0}).}$$

\begin{remark}
Even if $K_{1,0}(C,L',K_C)\ne0$, we nevertheless have $V\otimes \bb{C}\cdot\omega\subset K_{1,0}(X_0,L_0;\omega_{X_0}).$
\end{remark}


So let $\sigma\in V$. By the description of the derivative complex at the beginning of this section, to compute 
$$D(\mu_t)|_{t=0}(\sigma\otimes\omega),$$
we have to lift $\sigma\otimes\omega$ to first order in $t$, apply the Koszul differential $\mu_t$ to the lifting, then restrict the outcome divided by $t$ to $X_0$.

So let $\tilde{\sigma}$, $\tilde{\omega}$ be sections of $\cal{L}$ and $\omega_{\cal{X}/\Delta}$ extending $\sigma$ and $\omega$ respectively.

Since $\tilde{\sigma}$ vanishes on $E$ and $\tilde{\omega}$ vanishes on $C$, we can write 
\begin{eqnarray}\label{4.1}\tilde{\sigma}=\tilde{\sigma}'s_{E}
\end{eqnarray}
and 
\begin{eqnarray}\label{4.2}\tilde{\omega}=\tilde{\omega}'s_{C},
\end{eqnarray}
where $s_{E}$ (resp. $s_{C}$) is a section of $\cal{O}_{\cal{X}}(E)$ (resp. $\cal{O}_{\cal{X}}(C)$) vanishing precisely on $E$ (resp. $C$), and $\tilde{\sigma}'$ (resp. $\tilde{\omega}'$) are global sections of 
 
$$M:=\cal{L}(-E)|_{X_0}\cong\cal{L}(C)|_{X_0}$$
$$(\text{resp}\ N:=\omega_{\cal{X}/\Delta}(-C)|_{X_0}\cong\omega_{\cal{X}/\Delta}(E)|_{X_0})$$
 
 Notice that tensoring $\cal{L}$ by $\cal{O}_{\cal{X}}(-E)$ will increase the degree by $1$ on the $E$ component and decrease the degree by $1$ on the $C$ component. The line bundles $M$ and $N$ are described by the figure below. Notice that $M\otimes N\cong \omega_{X_0}\otimes L_0$.
 
 \vspace{.5cm}

\begin{figure}[h]
\psfrag{L'(-u)}{ $L'(-u)$}
\psfrag{O(v+u)}{$\cal{O}_E(u+v)$}
\psfrag{KC(2u)}{$K_C(2u)$}
\psfrag{KE}{$K_E$}
\psfrag{M}{$M$}
\psfrag{N}{$N$}

   \includegraphics[scale=0.6]{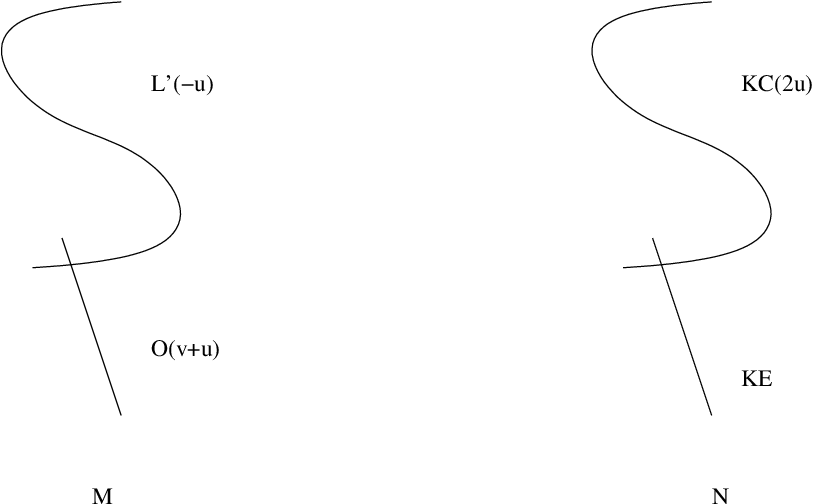}
  \caption{The twisted line bundles on the central fiber }
  \label{figure2}
\end{figure}

\vspace{2cm}

 By the construction of the derivative complex,
 
  \begin{eqnarray}\label{baby case}
D(\mu_t)|_{t=0}(\sigma\otimes\omega)&=&
\frac{\tilde{\sigma}\cdot\tilde{\omega}}{t}|_{X_0}\nonumber\\&=&\frac{(\tilde{\sigma}'s_E)\cdot(\tilde{\omega}'s_C)}{t}|_{X_0}
\nonumber\\&=&(\tilde{\sigma}'\tilde{\omega}')|_{X_0}\ \text{mod} \ \Im{\mu_0}
\end{eqnarray}

The general case is just notationally more complicated.

 Let
$$\sigma_{i_1}\wedge...\wedge\sigma_{i_{r-p}}\otimes\omega\in\bigwedge^{r-p}V\otimes H^0(E,K_{E}),$$
we will compute its image under $D(\delta_t)|_{t=0}$. 
Similar to the baby case, we have to lift $\sigma_{i_1}\wedge...\wedge\sigma_{i_{r-p}}\otimes\omega$ to first order, apply the Koszul differential $\delta_t$ to the lifting, then restrict the outcome divided by $t$ to $X_0$.

To this end,  write 
 $\tilde{\sigma}_{i_k}$, $\tilde{\omega}$ be sections of $\cal{L}$ and $\omega_{\cal{X}/\Delta}$ extending $\sigma_{i_k}$ and $\omega$ respectively. 
  Since $\tilde{\sigma}_{i_k}$ vanishes on $E$ and $\tilde{\omega}$ vanishes on $C$, we can write 
\begin{eqnarray}\label{4.1}\tilde{\sigma}_{i_k}=\tilde{\sigma}'_{i_k}s_{E}
\end{eqnarray}
and 
\begin{eqnarray}\label{4.2}\tilde{\omega}=\tilde{\omega}'s_{C},
\end{eqnarray}
 as before.

We compute

\begin{eqnarray}\label{general}
D(\delta_t)|_{t=0}(\sigma_{i_1}\wedge...\wedge\sigma_{i_{r-p}}\otimes\omega)&=&
\frac{\delta(\tilde{\sigma}_{i_1}\wedge...\wedge\tilde{\sigma}_{i_{r-p}}\otimes\tilde{\omega})}{t}|_{X_0}\nonumber\\&=&\sum_{k=1}^{r-p}(-1)^k\frac{\tilde{\sigma}_{i_1}\wedge...\wedge\widehat{\tilde\sigma}_{i_k}\wedge...\wedge\tilde\sigma_{i_{r-p}}\otimes(\tilde{\sigma}'_{i_k}s_{E})(\tilde{\omega}'s_{C})}{t}|_{X_0}
\nonumber\\&=&\sum_{k=1}^{r-p}(-1)^k\sigma_{i_1}\wedge...\wedge\widehat{\sigma}_{i_k}\wedge...\wedge\sigma_{i_{r-p}}\otimes(\tilde{\sigma}'_{i_k}\tilde{\omega}')|_{X_0}\ \text{mod} \ \Im{\delta_0}
\end{eqnarray}

\section{The study of obstruction classes}

As explained in the introduction, our goal is to show that the rank of the obstruction map
$$\xymatrix{K_{r-p,0}(X_0,L_0;\omega_{X_0})\ar[r]^-{D(\delta_t)|_{t=0}}&K_{r-p-1,1}(X_0,L_0;\omega_{X_0}).}$$ is as big as it could be as this would imply $K_{r-p,0}(X_t,L_t;K_{X_t})$ is as small as it could be for $t\ne0$. 

Again let us analyze the baby case $p=r-1$ first. By a result of \cite{AS}, the multiplication map
 $$\xymatrix{H^0(X_t,L_t)\otimes H^0(X_t,\omega_{X_t})\ar[r]^-{\mu_t}&H^0(X_t,\omega_{X_t}\otimes L_t)}.$$
 is already surjective for the general fiber, which implies 
 $K_{1,0}(X_t,L_t;K_{X_t})$ is of expected dimension.  So we are not proving anything new here. But it is helpful to redo this case via infinitesimal methods, because the later has the potential to generalize.
 
First notice that 
 $$\xymatrix{H^0(X_0,L_0)\otimes H^0(X_0,\omega_{X_0})\ar[r]^-{\mu_0}&H^0(X_0,\omega_{X_0}\otimes L_0)}.$$
 is not surjective. The problem is that any section in $H^0(X_0,\omega_{X_0})$ vanishes at $u$, but there is a section in $H^0(X_0,\omega_{X_0}\otimes L_0)$ not vanishing at $u$. Moreover, $\mu_0$ is exactly corank one. This is because on the $E$ component, $\mu_0$ becomes
 $$\xymatrix{H^0(\cal{O}_E(v))\otimes H^0(\cal{O}_E(u))\ar[r]&H^0(\cal{O}_E(u+v))},$$
 which is of corank 1 (c.f. figure \ref{figure1}). Since by \cite{AS} (or by induction hypothesis if one wants a proof independent of \cite{AS}), the map
 $$H^0(C,L')\otimes H^0(K_C)\longrightarrow H^0(L'\otimes K_C)$$
 is surjective, we see that if a section  $\tau\in H^0(X_0, \omega_{X_0}\otimes L_0)$ vanishes at $u$, then it is in the image of $\mu_0$. Therefore $k_{1,0}(X_t,L_t;K_{X_t})$ jumps up by one at $t=0$.
 
 Now by the computation of the obstruction class in (\ref{baby case}),
 $$(\tilde{\sigma}'\tilde{\omega}')|_{X_0}$$ 
 is in the image of 
 $$H^0(X_0,M)\otimes H^0(X_0,N)\longrightarrow H^0(X_0,\omega_{X_0}\otimes L_0)$$
 Since there are always sections in $H^0(X_0,M)$ and $H^0(X_0,N)$ not vanishing at $u$, we can easily choose $\tilde{\sigma}'$ and $\tilde{\omega}'$ such that $(\tilde{\sigma}'\tilde{\omega}')|_{X_0}$ does not vanish at $u$ (Notice that any (global) sections of $M$ and $N$ will extend to nearby fiber.) Therefore there is at least one dimensional subspace of $K_{1,0}(X_0,L_0;K_{X_0})$ does not deform to nearby fiber, namely $(\tilde{\sigma}'s_E)(\tilde{\omega}'s_C)$. This means $K_{1,0}(X_t,L_t;K_{X_t})$ is of expected dimension for $t\ne0$. This proves the baby case.
 
The cases for general $p$ is much more delicate. 
 There are two possible ways to show the obstruction classes in (\ref{general}) is not in the image of $\delta_0$.

The easier way is to mimic the baby case to show $(\tilde{\sigma}'_{i_k}\tilde{\omega}')|_{X_0}$ does not lie in the image of
$$\xymatrix{H^0(X_0,L_0)\otimes H^0(X_0,\omega_{X_0})\ar[r]^-{\mu_0}&H^0(X_0,\omega_{X_0}\otimes L_0)}.$$
(As we have seen before, $\mu_0$ is of corank one).
This will be the case if  $(\tilde{\sigma}'_{i_k}\tilde{\omega}')|_{X_0}$ does not vanish at $u$. Then the obstruction class in (\ref{general}) has no chance to be in $\Im(\delta_0)$.


To make this idea more precise, choose a basis $\{\sigma_1,...,\sigma_r\}$ of $V$ adapted to $u$ i.e. $\sigma_k|_C$ vanishes to order exactly $k$ along $u$ (therefore $\sigma_k|_E=0$ for $k\ge1$).  Use the same notation as (\ref{4.1}) and (\ref{4.2}), we have
$$(\tilde\sigma'_1\tilde\omega')|_{X_0}$$
is not in the image of $\mu_0$ because $\sigma_1|_C$ vanishes to order exactly $1$ at $u$, any extension $\tilde{\sigma}_1=\tilde{\sigma}_1'\cdot s_E$ we choose would have $\tilde{\sigma}_1'$ does not vanish at $u$ (Because $s_E|_C$ vanishes to order $1$ at $u$, so $\tilde{\sigma}_1'$ does not vanish.). Similarly the extension $\tilde{\omega}'$ does not vanish at $u$. (Although the choice of extensions is not unique, different choices give the same obstruction class modulo $\Im(\delta_0)$.)

 Howerever for $k\ge 2$, because $\sigma_k|_C$ vanishes to order at least $2$ at $u$, we could choose suitable extension $\tilde\sigma_k$ (modulo $Im(\delta_0)$ this does not depend on the choice of extension) such that $\tilde\sigma_k=\tilde\sigma_k''s_E^2$ and therefore
$$(\tilde\sigma_k'\tilde\omega')|_{E}=(\tilde\sigma_k''s_E\tilde\omega')|_E=0.$$

Thus for any ${\bf{1=i_1}}< i_2<i_3<...<i_{r-p}\le r$,
$$D(\delta_t)|_{t=0}({\bf\sigma_{1}}\wedge\sigma_{i_2}\wedge...\wedge\sigma_{i_{r-p}}\otimes\omega)=-\sigma_{i_2}\wedge...\wedge\sigma_{i_{r-p}}\otimes(\tilde\sigma'_1\tilde\omega')|_{X_0}+\sum_{k=2}^{r-p}(-1)^k\sigma_{i_1}\wedge...\wedge\widehat{\sigma}_{i_k}\wedge...\wedge\sigma_{i_{r-p}}\otimes(\tilde{\sigma}'_{i_k}\tilde{\omega}')|_{X_0},$$
by looking at its restriction to $E$, we see immediately that 
$$\{D(\delta_t)|_{t=0}({\bf\sigma_{1}}\wedge\sigma_{i_2}\wedge...\wedge\sigma_{i_{r-p}}\otimes\omega)|\ 2\le i_2<i_3<...<i_{r-p}\le r\}$$
 are linearly independent in $K_{r-p-1,1}(X_0,L_0;\omega_{X_0})$.

Thus at this point the rank of $D(\delta_t)|_{t=0}$ is at least 
$${r-1\choose p},$$
and therefore 

\begin{eqnarray}\label{atleast}k_{r-p,0}(X_t,L_t;\omega_{X_t})\le {r\choose p}-{r-1\choose p}={r-1\choose p-1}
\end{eqnarray}
for $t\ne0$.

The second way to show obstructions are non-trivial is more delicate. As we have already seen, for $2\le i_1<...<i_{r-p}\le r$, restricting to $E$ does not give any information to $D(\delta_t)|_{t=0}(\sigma_{i_1}\wedge...\wedge\sigma_{i_{r-p}}\otimes\omega)$ since they all restrict to zero on $E$. We will have to study the restriction of 
$D(\delta_t)|_{t=0}(\sigma_{i_1}\wedge...\wedge\sigma_{i_{r-p}}\otimes\omega)$  to $C$.

 (\ref{general}) restricted to $C$ becomes
\begin{eqnarray}\label{4.4}
D(\delta_t)|_{t=0}(\sigma_{i_1}\wedge...\wedge\sigma_{i_{r-p}}\otimes\omega)|_C&=&\sum_{k=1}^{r-p}(-1)^k\sigma_{i_1}\wedge...\wedge\widehat{\sigma}_{i_k}\wedge...\wedge\sigma_{i_{r-p}}\otimes(\tilde{\sigma}'_{i_k}\tilde{\omega}')|_{C}\ \text{mod} \ \Im{\delta_0}\nonumber
\end{eqnarray}

Here $\tilde{\sigma}_{i_k}'|_C\in H^0(C,L'(-u))$ and is equal to $\sigma_{i_k}$ for $i_k\ge1$ if we abuse the notation by thinking of $\sigma_{i_k}$ as sections of $L'(-u)$ instead of $L'$. (Thus $\sigma_1$ is a section of $L'(-u)$ which does not vanish at $u$ and $\sigma_2$ vanishes to order $1$ at $u$, etc.) On the other hand, $\tilde{\omega}'|_C\in H^0(K_C(2u))$ and does not vanish at $u$, denote it by $\omega'$. With the notation above, the obstruction class becomes
\begin{eqnarray}\label{4.4}
\sum_{k=1}^{r-p}(-1)^k\sigma_{i_1}\wedge...\wedge\widehat{\sigma}_{i_k}\wedge...\wedge\sigma_{i_{r-p}}\otimes(\sigma_{i_k}\omega')=\delta_0(\sigma_{i_1}\wedge...\wedge\sigma_{i_{r-p}}\otimes\omega').
\end{eqnarray}

\begin{remark} Here we are still using $\delta_0$ to denote the restriction to $C$ of the original Koszul differential $\delta_0$ on $X_0$. Equality (\ref{4.4}) does not mean $D(\delta_t)|_{t=0}(\sigma_{i_1}\wedge...\wedge\sigma_{i_{r-p}}\otimes\omega)|_C\in \Im(\delta_0)$, since $\omega'\notin H^0(K_C)$.

\end{remark}

Now the non-trivialness of obstruction classes on $X_0$ boils down to a question on $(C',L')$. This is 
\begin{theorem}\label{hypothesis}
Let $C$ be a general curve of genus $g$, $L'$ be a $g^r_d$ on $C$ such that $K_{r-p,0}(C,L';K_C)=0$ and $\{\sigma_0,...,\sigma_r\}$ is a basis of $H^0(C,L')$ adapted to a  general point $u\in C$, and $\omega'\in H^0(C,K_C(2u))\smallsetminus H^0(C,K_C)$. 
Consider the obstruction classes
\begin{eqnarray}\label{obclass}\{\delta_0(\sigma_{i_1}\wedge...\wedge\sigma_{i_{r-p}}\otimes\omega')\ |\ 2\le i_1<...<i_{r-p}\le r\}\subset K_{r-p-1,1}(C,L';K_C).
\end{eqnarray}

\begin{enumerate}
\item If these classes
  are linearly independent in $K_{r-p-1,1}(C,L',K_C)$, then 
$$K_{r-p,0}(X_t,L_t;\omega_{X_t})\cong K_{p-1,2}(X_t,L_t)^{\lor}=0.$$ 
 
 \item On the other hand, if these classes span $K_{r-p-1,1}(C,L';K_C)$, 
then 
$$k_{r-p,0}(X_t,L_t;\omega_{X_t})\le {r-1\choose p-1}-k_{r-p-1,1}(C,L';K_C)=-b_{p+1}(X_t,L_t).$$
which implies $K_{p,1}(X_t,L_t)=0$.
\end{enumerate}
\end{theorem}
\begin{proof}  The hypothesis in case (a) implies that $D(\delta_t)|_{t=0}$ in (\ref{derivative}) is either injective , which means no elements in $K_{r-p,0}(X_0,L_0;\omega_{X_0})$ will extend to nearby. In case (b), the rank of  $D(\delta_t)|_{t=0}$ is 
$$k_{r-p-1,1}(C,L';K_C)+{r-1\choose p},$$
which implies that
$$k_{r-p,0}(X_t,L_t;\omega_{X_t})\le {r-1\choose p-1}-k_{r-p-1,1}(C,L';K_C)=-b_{p+1}(X_t,L_t).$$ 
 Therefore only a subspace of $K_{r-p,0}(X_0,L_0;\omega_{X_0})$ of correct dimension will extend to nearby fibers to first order.
\end{proof}

Now we give a sufficient condition for the obstruction classes in (\ref{obclass}) to be linearly independent. Consider the diagram of complexes
\begin{eqnarray}\label{canonical}\nonumber
\xymatrix{\wedge^{r-p+1}H^0(L')\otimes H^0(K_C\otimes L'^{-1})\ar@{^{(}->}[r]\ar@{=}[d]&\wedge^{r-p}H^0(L')\otimes H^0(K_C)\ar[r]^-{\delta_0}\ar@{^{(}->}[d]^-{\alpha}&\wedge^{r-p-1}H^0(L')\otimes H^0(K_C\otimes L')\ar[d]\\
\wedge^{r-p+1}H^0(L')\otimes H^0(K_C\otimes L'^{-1})\ar@{^{(}->}[r]&\wedge^{r-p}H^0(L')\otimes H^0(K_C(2u)))\ar[r]^-{\delta}&\wedge^{r-p-1}H^0(L')\otimes H^0(K_C\otimes L'(2u))}
\end{eqnarray}

\begin{lemma}\label{lemma4.2}
Notation and assumption same as Theorem \ref{hypothesis}, if the bottom row of the above diagram is exact in the middle (the first row is exact by assumption of Theorem \ref{hypothesis}), then the obstruction classes in (\ref{obclass})
are linearly independent modulo $\Im(\delta_0)$. As a consequence of Theorem \ref{hypothesis}, $K_{r-p,0}(X_t,L_t;\omega_{X_t})=0$.

\end{lemma}

\begin{proof}
The assumption implies that $\Ker(\delta)=\Ker(\delta_0)\cong \wedge^{r-p+1}H^0(L')\otimes H^0(K_C\otimes L'^{-1})$. If a linear combination of the obstruction classes $\delta_0(\sigma_{i_1}\wedge...\wedge\sigma_{i_{r-p}}\otimes\omega')$ is equal to $\delta_0(c)$ for some $c\in \wedge^{r-p}H^0(L')\otimes H^0(K_C)$, then the same linear combination of the $\{\sigma_{i_1}\wedge...\wedge\sigma_{i_{r-p}}\otimes\omega'\}$ minus $\alpha(c)$ is in $\Ker(\delta)=\Ker(\delta_0)$. This contradicts with the fact that 
$$\{\sigma_{i_1}\wedge...\wedge\sigma_{i_{r-p}}\otimes\omega'\}$$
are linearly independent in $\wedge^{r-p}H^0(L')\otimes H^0(K_C(2u))$ modulo image of $\alpha$. 

\end{proof}

To end this section, we give a proof of Theorem \ref{generalcase}.
\begin{proof} of ${\bf Theorem \ref{generalcase}}.$) There are two cases
 \begin{enumerate}
\item
 $K_{p,1}(C,L')=0$.  By Proposition \ref{prop3.1}, $K_{p,1}(X_0,L_0)=0$ and ${\bf GV}(p)^r_{g+1,d+1}$ follows from upper-semicontinuity of Koszul cohomology. 
\item $K_{p-1,2}(C,L')\cong K_{r-p,0}(C,L,;K_C)^{\lor}=0$.

Starting from the defining sequence for the kernel bundle $M_{L'}$:
$$\xymatrix{0\ar[r]&M_{L'}\ar[r]&H^0(L')\otimes\cal{O}_C\ar[r]&L'\ar[r]&0},$$
taking $(r-p)$-th wedge, twisting by $K_C$ (resp. $K_C(2u)$) and then taking global sections, we get  
$$\xymatrix{0\ar[r]&\wedge^{r-p}M_{L'}\otimes K_C\ar[r]&\wedge^{r-p}H^0(L')\otimes H^0(K_C)\ar[r]^-{\delta_0}&\wedge^{r-p-1}H^0(L')\otimes H^0(K_C\otimes L')\ar[r]&0}$$
and therefore
\begin{eqnarray}\label{ker0}\Ker(\delta_0)=H^0(C,\wedge^{r-p}M_{L'}\otimes K_C).
\end{eqnarray}
Similarly,
\begin{eqnarray}\label{ker}\Ker(\delta)=H^0(C,\wedge^{r-p}M_{L'}\otimes K_C(2u)).
\end{eqnarray}
If
$$h^0(C,\wedge^{r-p}M_{L'}\otimes K_C)=h^0(C,\wedge^{r-p}M_{L'}\otimes K_C(2u)),$$
we conclude that 
$$\Ker(\delta_0)\cong\Ker(\delta)$$ which implies $K_{r-p,0}(X_t,L_t;\omega_{X_t})\cong K_{p,2}(X_t,L_t)^{\lor}=0$ by Lemma \ref{lemma4.2}.

\end{enumerate}
 \end{proof}

\section{Some Applications to the Maximal Rank Conjecture}
 
    
    In the case $p=1$, we can reduce condition (\ref{h0condition}) in Theorem \ref{generalcase} to a statement about the tangential variety $TC$ of $C$, namely, the existence of  a quadric containing $C$ but not containing $TC$. The condition on the tangential variety is quite interesting in its own right. Theorem \ref{mainresult} follows immediately from Theorem \ref{generalcase} and Lemma \ref{lemma5.1}.




\begin{lemma}\label{lemma5.1} For a general $L'=g^r_d$ on a general curve $C$ of genus $g$ with $K_{0,2}(C,L')=0$, (i.e. $\mu$ in (\ref{multimap}) is surjective), if there exists a quadric $Q\subset\bb{P}^r$ containing $\phi_{|L'|}(C)$ but not containing its tangential surface $TC:=\cup_{u\in C}T_uC\subset\bb{P}^r$,
then 
$$H^0(C,\wedge^{r-1}M_{L'}\otimes K_C)=H^0(C,\wedge^{r-1}M_{L'}\otimes K_C(2u)).$$

\end{lemma}

\begin{proof} Notice that
$$\wedge^{r}M_{L'}\cong L'^{-1}$$
and therefore
$$\wedge^{r-1}M^{\lor}_{L'}\cong M_{L'}\otimes L'.$$By Riemann-Roch, it suffices to show that
$$h^0(M_{L'}\otimes L'(-2u))= h^0(M_{L'}\otimes L')-2r.$$
The $\ge$ part is automatically true, only the $\le$ part needs to be proved.

We have diagram with exact rows
$$\xymatrix{0\ar[r]&H^0(M_{L'}\otimes L'(-2u))\ar[r]\ar@{^{(}->}[d]&H^0(L')\otimes H^0(L'(-2u))\ar[r]^-{\mu'}\ar@{^{(}->}[d]&H^0(L'^2(-2u))\ar@{^{(}->}[d]&\\
0\ar[r]&H^0(M_{L'}\otimes L')\ar[r]&H^0(L')\otimes H^0(L')\ar[r]^-{\mu}&H^0(L'^2)\ar[r]&0}$$
We need to show
$$\dim_{\bb{C}}\Ker(\mu')\le\dim_{\bb{C}}\Ker(\mu)-2r.$$
Let $H_u:=H^0(L')\otimes H^0(L'(-2u))$ and $\overline{H}_u$ be its image in 
$$\frac{H^0(L')\otimes H^0(L')}{\wedge^2H^0(L')}\cong S^2H^0(L').$$
$\overline{H}_u$ is the space of quadrics which contain the tangent line of $C$ at $u$.

We have 
$$\Ker(\mu')=\Ker(\mu)\cap H_u.$$
By hypothesis, $\overline{\Ker(\mu)}\nsubseteq \overline{H}_u$ for general $u$ (since $Q\notin \overline{H}_u$), then it follows that

$$\dim_{\bb{C}}(\overline{\Ker(\mu')})=\dim_{\bb{C}}(\overline{\Ker(\mu)\cap H_u})\le\dim_{\bb{C}}(\overline{\Ker(\mu)}\cap\overline{H}_u)\le\dim_{\bb{C}}(\overline{\Ker(\mu)})-1=:m-1.$$
Thus
\begin{eqnarray}
\dim_{\bb{C}}(\Ker(\mu'))&\le& m-1+\dim_{\bb{C}}(\wedge^2 H^0(L')\cap H_u)\nonumber\\&=&m-1+\dim_{\bb{C}}(\wedge^2H^0(L'(-2u)))\nonumber\\&=&m-1+{r-1\choose2}=m+{r+1\choose2}-2r\nonumber\\&=&\dim_{\bb{C}}(\Ker(\mu))-2r.\nonumber
\end{eqnarray}
\end{proof}

Now let us go the the proof of Corollary \ref{bestresult}. The numerical assumption in Corollary \ref{bestresult} turns out to be a technique assumption needed to verify assumption in Theorem \ref{mainresult} (b) about $TC$. This is equivalent to the numerical assumption in Lemma \ref{lemma5.2}.
 By appendix, if $L'$ is a general non-special $g^r_{2r-3}$ on a general curve $C$ of genus $r-3$, the number of quadrics containing $TC$ is at most
$${r-4\choose2}.$$

\begin{lemma}\label{lemma5.2} Let $C\subset\bb{P}^r$ be a general curve of genus $g$ embedded by a general $g^r_d$ $L'$ with $h^1(L')\le1$. Suppose
$${r+2\choose2}-(2d-g+1)>{r-4\choose2},$$
(i.e. the number of independent quadrics containing $C$ is at least ${r-4\choose2}$.)

 then there exists a quadric $Q$ on $\bb{P}^r$ containing $C$ but not containing $TC$. 
\end{lemma}

\begin{proof} Degenerate $(C,L')$ to $(C_0,L'_0)$ where $C_0$ is a nodal curve with two smooth components $Y$ and $Z$ meeting at a general point $u$. According to the value of $h^1(L')$, there are two cases.
\begin{enumerate}
\item $h^1(L')=0.$ 

$L'=g^r_{g+r}$ for $g\ge0$. If $0\le g\le r-3$, take $g_Y=0$, $g_Z=g$, $L'_0|_Y=\cal{O}_{\bb{P}^1}(r)$ and $L'_0|_Z=g^0_g$ (One could easily show such $(C_0,L'_0)$ can deform to $(C,L')$). Since there are only
$${r-2\choose2}$$ 
quadrics containing the tangential variety of the rational normal curve in $\bb{P}^r$ (see appendix) and in this range of $g$, the number of quadrics containing $C$ is at least
$${r+2\choose2}-(2d-g+1)={r+2\choose2}-(g+2r+1)>{r-2\choose2},$$
we conclude that there exists a quadric containing the nearby fiber $C$ but not containing $TC$. 

If $g>r-3$, we take $g_Y=r-3$, $g_Z=g-r+3$, $L_0'|_Y=g^r_{2r-3}$ (a general one) and $L_0'|_Z=g^0_{g-r+3}$. By proposition \ref{prop7.1} in the appendix, the number of quadrics containing $TC$ for nearby $C$ is at most ${r-4\choose2}$. By the numerical hypothesis, we get our conclusion.

\item $h^1(L')=1$. 

The argument is similar as above except that we need to deal with $L'=g^r_{g+r-1}$ for $g\ge r+1$. Again, if $r+1\le g\le 2r-2$, we take $g_Y=0$, $Z=C$, $L_0'|_Y=\cal{O}_{\bb{P}^1}(r)$ and $L_0'|_Z=L'(-ru)=g^0_{g-1}$. 

If $g>2r-2$, take $g_Y=r-3$, $g_Z=g-r+3$, $L_0'|_Y=g^r_{2r-3}$, $L_0'|_Z=g^0_{g-r+2}$. Here $L_0'|_Z$ comes from a general $g^r_{g+2}$ on $Z$ twisted by $\cal{O}_Z(-ru)$. The rest of the argument is exactly the same as in case (a).
\end{enumerate}

\end{proof}

\begin{proof} of {\bf Corollary \ref{bestresult})}
 First notice that by Corollary \ref{cor2.2}, to show projective normality of a general pair, it suffices to show (\ref{multimap}) is surjective. We will fix $h^1$ and $r$ and do induction on $g$.
 
  For $h^1=0$ case, we start with the fact that rational normal curve is projectively normal, (i.e. $(MRC)^r_{0,r}$ holds). For $h^1=1$ case, we use the fact that general canonical curve is projectively normal (i.e. $(MRC)^r_{r+1,2r}$ holds). Now assuming $(MRC)^r_{g,d}$ holds, by Lemma \ref{lemma5.2}, as long as 
 \begin{eqnarray}\label{condition}{r+2\choose2}-(2d-g+1)>{r-4\choose2},
 \end{eqnarray} 
  Theorem \ref{mainresult} (b) is satisfied, which implies $(MRC)^r_{g+1,d+1}$ (which is equivalent to projective normality). Plug in $d=g+r-h^1$ to (\ref{condition}), we immediately get the bound on $d$ as in the statement of the Theorem.

\end{proof}

\begin{proof} of {\bf Corollary 1.8)} The case $r=1,2$ is trivial. The arguments for $r=3,4$ are completely similar, so we will only prove the case $r=4$. Again we do induction on $g$. First suppose we have proved $(MRC)$ for the base cases $g=5h^1$, $L=g^4_{4h^1+4}$. Then notice that for $r=4$, ${r-4\choose2}=0$ and therefore there is no quadric containing the tangential variety in Theorem \ref{mainresult} (b). Thus $(MRC)^4_{g,d}$ implies $(MRC)^4_{g+1,d+1}$. It remains to prove $(MRC)$ for the base cases. When $h^1\le1$ $(MRC)^4_{5h^1,4h^1+4}$ is clear. If $h^1\ge 2$, we need to show $\mu$ in (\ref{multimap}) is injective. For $h^1=2$, $(MRC)^4_{10,12}$ is well known and is proved in \cite{FP}. If $h^1\ge3$, we  we could degenerate again to $C_0=Y\cup Z$ with $g_Y=10$, $g_X=5h^1-10$, $L_0|_Y=g^4_{12}$, $L_0|_Z=g^0_{4h^1-8}=g^4_{4h^1-4}(-4u)$, again it is easy to check such $(C_0,L_0)$ is smoothable in $\cal{G}^4_{5h^1,4h^1+4}$ (c.f \cite{W1} corollary 6.1 for details). The injectivity of $\mu$ in this case follows from the same argument as in Proposition \ref{prop3.1}. 
\end{proof}

It was also proved in \cite{F} that for any integer $s\ge1$, $(MRC)^{2s}_{s(2s+1),2s(s+1)}$ holds. In this case, $\rho=0$ and $h^1=s$. Thus by Theorem \ref{mainresult} (a), we have

\begin{corollary} $(MRC)^{2s}_{s(2s+1)+k,2s(s+1)+k}$ holds for all $s\ge1$, $k\ge0$, i.e. $(MRC)$ holds if $r=2h^1$. 
\end{corollary}

\section{Higher syzygies}
As we mentioned in the introduction,
the difficulty to generalize the inductive argument to higher syzygies is due to the lack of known cases to start the induction with and no analog of Theorem \ref{mainresult} for higher syzygies. Nevertheless, we collect some vanishing results we can obtain in this section. 

\begin{prop} \label{notsharp}For a general $g^r_d$ $L$ on a general curve $X$ with $g\ge r+1$, $K_{p,1}(X,L)=0$ for $p\ge\lfloor\frac{r+1}{2}\rfloor$. \end{prop}

\begin{proof} Start with the case $g_X=r+1$, $L'=K_X$, thanks to Voisin's solution to the generic Green conjecture, $K_{p,1}(X,K_X)=0$ for  $p\ge \lfloor\frac{r+1}{2}\rfloor$. If $g>r+1$, we degenerate to $X_0=Y\cup Z$ with $g_Y=r+1$, $g_Z=g-r-1$, $L_0|_Y=K_Y$, $L_0|_Z=g^0_{d-2r}=g^r_{d-r}(-ru).$ The statement then follows from the same argument as in proposition \ref{prop3.1}.  

\end{proof}

\begin{remark} Using the same degeneration as in proposition \ref{notsharp}, we also have
$$k_{p,1}(X,L)\le k_{p.1}(Y,K_Y)=[{r-1\choose p}-{r-1\choose p-1}]r+{r+1\choose p}-{r+1\choose p+1}$$
 for $1\le p< \lfloor\frac{r+1}{2}\rfloor$. 
 We will improve this bound using an infinitesimal argument.

\end{remark}

Even though we do not have an analog of Theorem \ref{mainresult}, when $g$ is not too big compared to $p$, the inductive argument still go through:  
\begin{lemma}\label{lemma6.2} Let $C$ be a general curve of genus $g$ and $L'$ be a $g^r_d$ with $h^1(L')=1$. If $p\le r-\lfloor\frac{g+1}{2}\rfloor$, then the sequence 
$$\xymatrix{\wedge^{r-p+1}H^0(L')\otimes H^0(K_C\otimes L'^{-1})\ar@{^{(}->}[r]&\wedge^{r-p}H^0(L')\otimes H^0(K_C(2u)))\ar[r]^-{\delta}&\wedge^{r-p-1}H^0(L')\otimes H^0(K_C\otimes L'(2u))}$$
is exact in the middle.
\end{lemma}

\begin{proof} Let $t\in H^0(K_C\otimes L'^{-1})$ be a generator. Multiplication by $t$ gives an embedding 
$$\xymatrix{H^0(L')\ar[r]^-{\cdot t}&H^0(K_C)}.$$
Denote its image $W$. Let $C'$ be the image of $C$ under the map given by $|K_C(2u)|$. $C'$ is of arithematic genus $g+1$ and has a cusp. We can identify $H^0(C,K_C(2u))$ with $H^0(C',\omega_{C'})$, where $\omega_{C'}$ is the dualizing sheaf of $C'$. With the above notation,
we can identify the cohomology group in question with $K_{r-p,1}(C',\omega_{C'};W)\subset K_{r-p,1}(C',\omega_{C'})$. Since curves in $K3$ surfaces satisfies Green conjecture (c.f. \cite{V1}, \cite{V2}), by degenerating $C'$ to a cuspidal curve in $K3$ surface, we have
$$K_{r-p,1}(C',\omega_{C'})=0$$
for $r-p\ge \lfloor\frac{g+1}{2}\rfloor$.

\end{proof}
Again we start our induction with a general curve of genus $r+1$, and $L'=K_C$. For $p< \lfloor\frac{r+1}{2}\rfloor$, we have
$$K_{p-1,2}(C,K_C)^{\lor}\cong K_{r-p,1}(C,K_C)=0$$

Now we apply the construction in section $3$, by Lemma \ref{lemma4.2} and Lemma \ref{lemma6.2}, we get

\begin{prop}\label{notnew}For a general $g^r_d$ L on a general curve $X$ with $h^1(L)=1$, 
\begin{enumerate}
\item $k_{p-1,2}(X,L)=k_{r-p,0}(X,L;K_X)=0$ if $p\le r-\lfloor\frac{g}{2}\rfloor$.
\item $k_{p-1,2}(X,L)\le(g-2r+2p-1){r-1\choose p-1}$ if $p>r-\lfloor\frac{g}{2}\rfloor$.
\end{enumerate}

\end{prop}

\begin{proof}Again we start our induction on $g$ with a general curve of genus $r+1$, and $L'=K_C$.  We always have
$$K_{p-1,2}(C,K_C)^{\lor}\cong K_{r-p,1}(C,K_C)=0.$$
for $r-p\ge \lfloor\frac{r+1}{2}\rfloor$.

Now we apply the construction in section $3$. If $\lfloor\frac{g}{2}\rfloor\le r-p$, Lemma \ref{lemma4.2} and  \ref{lemma6.2} applies and we get (a).

When $\lfloor\frac{g}{2}\rfloor$ get passed $r-p$ (or equivalently $g>2r-2p+1$), we nevertheless have estimate (\ref{atleast}) for each attached elliptic tail. Thus the bounds in (b) follows.

\end{proof}
  Combining the results of Propositions \ref{notsharp} and \ref{notnew}, we get Theorem \ref{badresult}.
  
\begin{remark} For line bundles with $h^1=1$ the assumption $p\le r-\lfloor\frac{g}{2}\rfloor$ is equivalent to the condition that
$d\ge2g-2+p-\lfloor\frac{g-1}{2}\rfloor$. Thus Theorem \ref{notnew} is the generic version of the generalized Green-Lazarsfeld conjecure for special linear series (c.f \cite{GL1}). However, this generic version is known to follow from the generic Green conjecture (c.f \cite{AF} proposition 4.30). It seems to the author that the bound in (b) is new.

\end{remark}

\section{Appendix}
In this appendix, we prove the following statement, which is needed in the proof of Theorem \ref{bestresult}.

\begin{prop}\label{prop7.1} For a general curve $C$of genus $r-3$ embedded in $\bb{P}^r$ by a general $g^r_{2r-3}$, the number of quadrics containing $TC$ is at most
$${r-4\choose 2}.$$
\end{prop}

Consider the rational normal curve $C$ of degree $d$ in $\bb{P}^d$.
It is well known that there are 
$$d\choose2$$
independent quadrics containing $C$. Denote them $\Delta_{a,b}$ for $0\le a<b\le d-1$, where $\Delta_{a,b}$ is the $2\times2$ minor corresponding to columns $a$ and $b$ of the matrix
$$\left(\begin{matrix}x_0&x_1&x_2&...&...&x_{d-2}&x_{d-1}\\x_1&x_2&x_3&...&...&x_{d-1}&x_d\end{matrix}\right)$$
with the usual convention that $\Delta_{a,b}=-\Delta_{b,a}$.

It is proved in \cite{E} that there are 
$$d-2\choose2$$
quadrics containing $TC$. They are
$$\Gamma_{a,b}=\Delta_{a+2,b}-2\Delta_{a+1,b+1}+\Delta_{a,b+2}$$
for $0\le a,b\le d-3$.

Now consider the projection $C'$ of $C$ to $\bb{P}^r$ ($d=2r-3$, $r\ge 3$) given by: 
$$t\longrightarrow[1,t^2,t^4,...,t^{2r-6},t^{2r-5},t^{2r-4},t^{2r-3}].$$
$C'$ has arithmetic genus $r-3$ and has a unique singular point at $t=0$ locally isomorphic to $Spec(\bb{C}[t^2,t^{2r-5}])$.

\begin{lemma}\label{lin}The complete linear system $|\mathcal{O}_{C'}(1)|$ has projective dimension $r$, i.e. $C'\subset\bb{P}^r$ is linearly normal. As a consequence, $C'$ is smoothable in $\bb{P}^r$.
\end{lemma}
\begin{proof}Denote $L_k=span\{P_1,P_3,...,P_{2k-1}\}\subset\mathbb{P}^{2r-3}$, where $P_i=[0,0,...,1^{(ith)},...,0]$, $i=0,1,...,r$ and  $C_k\subset\mathbb{P}^{2r-3-k}$ the projection of $C$ with center $L_k$. The curve $C_k$ has a unique singular point locally isomorphic to $Spec(\bb{C}[t^2,t^{2k+1}])$. Note that $C'=C_{r-3}\subset\bb{P}^r$. We use induction to show that the complete linear system $\mathcal{O}_{\bb{P}^{2r-3-k}}(1)|_{C_k}$ has projective dimension $2r-3-k$. The natural projection map
$Pr_k:C_k\rightarrow C_{k+1}$ induces an inclusion $H^0(\cal{O}_{C_{k+1}}(1))\subset H^0(\cal{O}_{C_k}(1))$. By induction hypothesis, $h^0((\cal{O}_{C_k}(1))=h^0(\cal{O}_{\bb{P}^{2r-3-k}}(1))=2r-2-k$. Since we obtain $C_{k+1}$ from $C_k$ by projection from a point, $h^0((\cal{O}_{C_{k+1}}(1))\ge h^0((\cal{O}_{C_k}(1))-1$. Since $C_{k+1}$ has arithmetic genus one higher than $C_k$,  $H^0(\cal{O}_{C_{k+1}}(1))\subsetneqq H^0(\cal{O}_{C_k}(1))$. Thus $h^0((\cal{O}_{C_{k+1}}(1))=2r-3-k$. For the last statement, note that the curve $C'$ only has plane curve singularity, thus is smoothable (as an abstract curve). Moreover, since $h^0(\cal{O}_{C'}(1))=r+1$, $\cal{O}_{C'}(1)$ is a complete non-special $g^r_{2r-3}$. For any one parameter smoothing $(C_t,L_t)$ of the pair $(C',\cal{O}_{C'}(1))$, since $h^0$ of the central fiber does not jump up, all $r+1$ global sections of $\cal{O}_{C'}(1)$ deform to $L_t$.
\end{proof}

\begin{proof} of {\bf Proposition 8.1}
We could explicitly compute the quadrics containing $TC'$: they are just quadrics in $\bb{P}^{2r-3}$ containing $TC$ with singular locus containing the center of projection $L_{r-3}=span\{P_1,P_3,...,P_{2r-7}\}$. 
Now if we think of each quadric $\Gamma_{a,b}$ as a $(2r-2)\times(2r-2)$ symmetric matrix, we are just looking for matrices $Q\in\ S_{\Gamma}:=span\{\Gamma_{a,b}|\ 0\le a<b\le 2r-6\}$ such that
$L_{r-3}\subset\Ker{Q}$ (We think of $Q$ as an linear operator on $\bb{C}^{2r-2}$ and $L_{r-3}$ as a subspace of $\bb{C}^{2r-2}$).

Notice that each $\Gamma_{a,b}$, as a matrix, only has possibly non-zero entries at $(i,j)$-spot ($i$ and $j$ go from $0$ to $2r-3$) if
$$i+j=a+b+3.$$ Said differently, each $\Gamma_{a,b}$, as a matrix, is supported on one of the diagonals.

For each $4\le k\le 4r-10$, there are $(\lfloor\frac{k}{2}\rfloor-1)$ $\Gamma_{a,b}$'s contributing to nonzero entries on the line
$$i+j=k,\ \text{for } 4\le k\le 2r-3,$$
and $(2r-4-\lfloor\frac{k+1}{2}\rfloor)$ $\Gamma_{a,b}$'s if $2r-3<k\le4r-10$. 

Write
$$S_{\Gamma}=\oplus_{k=4}^{4r-10}S_k$$
where $S_k=span\{\Gamma_{a,b}|\ 0\le a<b\le2r-6,\ a+b=k-3\}$. It is obvious that if $Q\in S_{\Gamma}$ vanishes on $L_{r-3}$, then its $S_k$ component also vanishes on $L_{r-3}$. Thus it suffices to count how many quadrics in each $S_k$ vanishes on $L_{r-3}$.

Let's just consider the case  $4\le k\le 2r-3$, the other case is similar.

When $k$ is odd, vanishing on $L_{r-3}$ imposes $(\frac{k-1}{2})$ independent conditions on $S_k$, more than dimension of $S_k$. Thus no quadrics in $S_k$ vanishes on $L_{r-3}$.

When $k$ is even, vanishing on $L_{r-3}$ only imposes $\lceil\frac{k}{4}\rceil$ independent conditions. We conclude that for  $4\le k\le 2r-3$, there are 
$$\sum_{8\le k\le 2r-4,\ k\ \text{even}}(\frac{k}{2}-1-\lceil\frac{k}{4}\rceil)=\lfloor\frac{r^2-8r+16}{4}\rfloor$$
quadrics containing $TC'$ (if $r\le 5$ there are none!).

Similarly, for $2r-3<k\le4r-10$, we count that there are
$$\lfloor\frac{r^2-8r+16}{4}\rfloor-\lfloor\frac{r-4}{2}\rfloor$$
quadrics containing $TC'$.

So we get total of 
$$\lfloor\frac{r^2-8r+16}{4}\rfloor+\lfloor\frac{r^2-8r+16}{4}\rfloor-\lfloor\frac{r-4}{2}\rfloor={r-4\choose2}$$
quadrics containing $TC'$.
By specializing to $C'$, we conclude our proof.
\end{proof}


 




\begin{thebibliography}{999}

\bibitem{A1} M. Aprodu, Green-Lazarsfeld tonality conjecture for a generic curve of odd genus. Int. Math. Res. Notices, {\bf63}, 3409-3414, 2004.

\bibitem{A2} M. Aprodu, Remarks on syzygies of d-gonal curves. Math. Res. Lett., {\bf12},387-400, 2005.

\bibitem{AC} E. Arbarello and M. Cornalba, Su una congettura di Petri, Comment. Math. Helvetici {\bf56}, 1-38, 1981.

\bibitem{ACGH} E. Arbarello,  M. Cornalba, P. Griffiths and J. Harris,  Geometry of Algebraic Curves, Volume  I, Springer Grundlehren {\bf267}, 1985. 

\bibitem{AF} M. Aprodu, G. Farkas, The Green conjecture for curves on arbitrary $K3$ surface, Compositio Math. {\bf147}, 839-851, 2011.

\bibitem{AN} M. Aprodu, J. Nagel, Koszul Cohomology and Algebraic Geometry, University Lecture Series, Volumn {\bf52}, American Mathematical Society, 2010.

\bibitem{AS} E. Arbarello, E. Sernesi, Petri's approach to the study of the ideal associated to a special divisor, Inventiones math. {\bf49}  99-119, 1978.

\bibitem{B} E. Ballico, On the minimal free resolution of general embedding of curves, Pacific J. Math. {\bf 172}, 315-319, 1996.


\bibitem{BE2} E. Ballico and P. Ellia, The maximal rank conjecture for nonspecial curves in $\mathbb{P}^n$, Math. Z. {\bf196}, 355-367, 1987.

\bibitem{BF1} E. Ballico and C. Fontanari, Normally generated line bundles on general curves, J. Pure and Applied Algebra. {\bf 214}, 837-840, 2010.


\bibitem{BF2} E. Ballico and C. Fontanari, Normally generated line bundles on general curves II, J. Pure and Applied Algebra. {\bf 214}, 1450-1455, 2010.



\bibitem{CU} F. Cukierman and D. Ulmer, Curves of genus ten on $K3$ surfaces, Compositio Mathematica {\bf89}, 81-90, 1993.

\bibitem{Ein} L. Ein,  A remark about the syzygies of generic canonical curve, J. Diff. Geom. {\bf26}, 361-367, 1987.

\bibitem{E} D. Eisenbud, Green's conjecture: an orientation for algebraists. Free resolutions in commutative algebra and algebraic geometry (Sundance, UT, 1990), 51Ð78, Res. Notes Math., 2, Jones and Bartlett, Boston, MA, 1992.

\bibitem{EH} D. Eisenbud and  J. Harris,   A simpler proof of the Gieseker-Petri theorem on special divisors, Invent. Math. {\bf74}, 269-280, 1983.

\bibitem{EH2} D. Eisenbud and  J. Harris, Divisors on general curves and cuspidal rational curves. Invent. Math. {\bf74}, 371-418, 1983.

\bibitem{EH1} D. Eisenbud and  J. Harris, Limit linear series: Basic theory, Invent. Math. {\bf85}, 337-371, 1986.

\bibitem{EH3} D. Eisenbud and  J. Harris, Irreducibility and monodromy of some families of linear series, Annales scientifiques de l'\'{E}cole Normale Sup\'{e}rieure, S\'{e}r.4, {\bf20}, 65-87, 1987.

\bibitem{FP} G. Farkas and M. Popa, Effective divisors on $\mathcal{M}_g$, curves on K3 surfaces, and the slope conjecture,  J. Algebraic Geom. {\bf14(2)}, 241--267, 2005.

\bibitem{F2} G. Farkas, Syzygies of curves and the effective cone of $\overline{\cal{M}_g}$, Duke Math. J. {\bf 135}, 53-98, 2006.

\bibitem{F} G. Farkas, Koszul divisors on moduli space of curves, Amer. J. Math. {\bf131}, 819-867, 2009.

\bibitem{G} D. Gieseker,  Stable curves and special divisors, Invent. Math. {\bf66}, 251-275, 1982.

\bibitem{Gr} M. Green, Koszul cohomology and the geometry of projective varieties, J. Differential Geom. {\bf19} 125-171, 1984.

\bibitem{GL84} M. Green and R. Lazarsfeld, The nonvanishing of certain Koszul Koszul cohomology groups, J. Diff. Geom. {\bf19}, 168-170, 1984.

\bibitem{GL1} M. Green and R. Lazarsfeld, On the projective normality of complete linear series on an algebraic curve, Invent. Math. {\bf83}, 73-90, 1986. 

\bibitem{GL} M. Green and R. Lazarsfeld, Deformation theory, generic vanishing theorems, and some conjectures of Enriques, Catanese and Beauville, Invent. Math. {\bf90}, 389-407, 1987.

\bibitem{H} J. Harris, Curves in projective space, Les Press de l'Universit$\acute{e}$ de Montr$\acute{e}$al, 1982. 
 
\bibitem{HM} J. Harris and I. Morrison,  Moduli of Curves, Graduate Text in Mathematics {\bf187}, Springer-Verlag New York, 1998.




\bibitem{W} J. Wang, Deformations of pairs $(X,L)$ when $X$ is singular, Proc. A.M.S. {\bf 140}(9), 2953-2966, 2012.


\bibitem{W1} J. Wang, On the projective normality of line bundles of extremal degree, Math. Ann {\bf 355}(3), 1007-1024, 2013.

\bibitem{V1} C. Voisin, Green's generic syzygy conjecture for curves of even genus lying on a $K3$ surface. J. Eur. Math.Soc., {\bf4}, 363-404, 2002.

\bibitem{V2} C. Voisin, Green's canonical syzygy conjecture for generic curves of odd genus. Compositio Math., {\bf141} (5), 1163-1190, 2005.

\end{thebibliography}
\end{document}